\documentclass[11pt]{amsart}
\usepackage{amsmath,amsfonts,amsthm,mathrsfs,amssymb,amscd,comment,enumerate,amsxtra,url,tikz-cd,physics,tcolorbox,graphicx}
\input{xypic}
\xyoption{all}
\usepackage{xcolor} 
\colorlet{mdtRed}{red!50!black}
\definecolor{dblue}{rgb}{0,0,.6}
\usepackage[colorlinks]{hyperref}
\hypersetup{linkcolor=blue,citecolor=dblue,filecolor=dullmagenta,urlcolor=mdtRed}

\newtcolorbox{mymathbox}[1][]{colback=white, sharp corners, #1}

\newtheorem{theorem}[equation]{Theorem}

\newtheorem{lemma}[equation]{Lemma}

\newtheorem{definition}[equation]{Definition}

\newtheorem*{theorem*}{Theorem}

\theoremstyle{remark}
\newtheorem{remark}[equation]{\bf Remark}

\newcommand{\Z}{\mathbb{Z}}
\newcommand{\C}{\mathbb{C}}

\renewcommand{\P}{\mathbb{P}}
\renewcommand{\O}{\mathcal{O}}

\newcommand{\mf}[1]{\mathfrak{#1}}

\newcommand{\mb}[1]{\mathbb{#1}}
\newcommand{\mc}[1]{\mathcal{#1}}

\newcommand{\hilb}{\mathfrak{h}ilb}

\newcommand{\I}{\mathscr{I}}
\newcommand{\ii}{\mathcal{I}}

\newcommand{\T}{\mathcal{T}}

\numberwithin{equation}{section}

\renewcommand \subsection[1]{
	\refstepcounter{equation}
	\refstepcounter{subsection}
	\noindent {\bf \arabic{section}.\arabic{subsection}.}{\bf #1}.
}

\makeatletter
\@namedef{subjclassname@2020}{\textup{2020} Mathematics Subject Classification}
\makeatother

\begin{document}

\title[Irreducibility of Some Nested Hilbert Schemes]
{Irreducibility of Some Nested Hilbert Schemes}

\author[C. Gangopadhyay]{Chandranandan Gangopadhyay} 

\address{Department of Mathematics, Indian Institute of Technology Bombay, Powai, Mumbai 
	400076, Maharashtra, India}

\email{chandra@math.iitb.ac.in} 

\author[P. Rasul]{Parvez Rasul} 

\address{Department of Mathematics, Indian Institute of Technology Bombay, Powai, Mumbai 
	400076, Maharashtra, India}

\email{rasulparvez@gmail.com} 

\author[R. Sebastian]{Ronnie Sebastian} 

\address{Department of Mathematics, Indian Institute of Technology Bombay, Powai, Mumbai 
	400076, Maharashtra, India}

\email{ronnie@math.iitb.ac.in} 



\begin{abstract}
	Let $S$ be a smooth projective surface over $\mb C$.
	Let $S^{[n_1,\dots,n_k]}$ denote the nested Hilbert scheme
	which parametrizes zero-dimensional subschemes 
	$\xi_{n_1} \subset \ldots \subset \xi_{n_k}$ 
	where $\xi_i$ is a closed subscheme of $S$ of length $i$.
	We show that $S^{[n,m]}$, $S^{[n,m,m+1]}$, $S^{[n,n+1,m]}$,
	$S^{[n,n+1,m,m+1]}$, $S^{[n,n+2,m]}$ and $S^{[n,n+2,m,m+1]}$
	are irreducible.
\end{abstract}

\maketitle
\section{Introduction}
Let $S$ be a smooth projective surface over $\mb C$.
The Hilbert scheme $S^{[n]}$ 
which parametrizes closed zero-dimensional subschemes of $S$ of length $n$  
is a well studied space.
It was shown by Fogarty in \cite[Theorem 2.4]{fo68} that the Hilbert scheme $S^{[n]}$
is a smooth projective variety of dimension $2n$.
A natural generalization of $S^{[n]}$ is the nested Hilbert 
Scheme, about which far less is known.
For an increasing tuple of positive integers $n_1< \ldots <n_k$,
the nested Hilbert scheme $S^{[n_1,\dots,n_k]}$ parametrizes nested zero-dimensional subschemes 
$\xi_{n_1} \subset \ldots \subset \xi_{n_k}$ 
where $\xi_i$ is a subscheme of $S$ of length $i$.
In recent years the nested Hilbert schemes $S^{[n,m]}$
have received growing attention. They have been studied
by several authors using techniques from commutative algebra,
representation theory and Lie algebras. In a recent 
article, \cite{RS21}, Ramkumar and Sammartano introduce methods to study
$S^{[n,m]}$. They use these methods to show that the scheme $S^{[2,n]}$ 
is smooth in codimension 3 and has rational singularities. 
In particular, $S^{[2,n]}$ is normal and Cohen-Macaulay.
They also mention several interesting questions related
to the schemes $S^{[n,m]}$, one of them being the irreducibility
of these schemes. 
The purpose of this article is to show that $S^{[n,m]}$
is irreducible. 

Before we state our results, we mention 
a few already existing results related to irreducibility
of nested Hilbert schemes.
The nested Hilbert scheme $S^{[1,n]}$ is irreducible 
of dimension $2n$ by \cite[Corollary 7.3]{Fo}.
The scheme $S^{[n,n+1]}$
is smooth and irreducible,
as shown in \cite[Theorem 3.0.1]{Cheah98}.
In \cite[Proposition 6]{GH04}, the authors show 
that $S^{[n,n+2]}$ is irreducible of dimension $2n+4$.
In \cite{Bulois-Evain}, Bulois and Evain 
studied irreducible components of nested Hilbert schemes
supported at a single point using
the connection between nested Hilbert schemes 
and commuting varieties of parabolic subalgebras.
In \cite[\S3.A]{Addington} the irreducibility of 
$S^{[n,n+1,n+2]}$ is proved. 
In \cite{RT22}, Ryan and Taylor
study the irreducibility, singularities and Picard groups of $S^{[n,n+1,n+2]}$.
In \cite[Theorem 3.1]{RS21}, Ramkumar and Sammartano have shown 
that $S^{[2,n]}$ is irreducible of dimension $2n$.

The following two results limit the collection of 
tuples $(n_1,\dots,n_k)$ for which the 
nested Hilbert scheme $S^{[n_1,\dots,n_k]}$ is irreducible.
By \cite[Corollary 3.17]{RT22} the nested Hilbert scheme
$S^{[n_1,\dots,n_k]}$ is reducible for $k>22$.
In \cite[Proposition 3.7]{RS21} the authors prove the existence of 
tuples $n_1 < \cdots < n_k$, for each $k \geqslant 5$, 
such that the nested Hilbert scheme 
$(\mb A^2)^{[n_1,\dots,n_k]}$ is reducible.
We refer the reader to 
\cite{RT22}, \cite{RS21} and the references therein
for more results related to the geometry of nested Hilbert schemes.

In \cite{RS21}, the authors pose the problem of irreducibility
of the two step nested Hilbert schemes, see \cite[Question 9.4]{RS21}.
Our goal in this paper is to prove the following results on 
irreducibility of nested Hilbert schemes.

\begin{theorem*}[Theorem \ref{main-theorem}]
	Let $n$ and $m$ be two positive integers such that $n < m$. 
	Then $S^{[n,m,m+1]}$ and $S^{[n,m]}$ are irreducible. 
\end{theorem*}

\begin{theorem*}[Theorem \ref{irreducible n,n+1,m}]
	Let $n$ and $m$ be two positive integers such that $n+1 < m$.
	Then $S^{[n,n+1,m,m+1]}$ and $S^{[n,n+1,m]}$ are irreducible.
\end{theorem*}

\begin{theorem*}[Theorem \ref{irreducible n+2}]
	Let $n$ and $m$ be two positive integers such that $n+2 < m$.
	Then $S^{[n,n+2,m,m+1]}$ and $S^{[n,n+2,m]}$ are irreducible.
\end{theorem*}
Let $E$ be a locally free sheaf on $S$ 
and let ${\rm Quot}(E,d)$ denote the Grothendieck Quot scheme 
of quotients of $E$ of length $d$. In \cite[Theorem 1]{EL99} it is proved 
that this Quot scheme is irreducible. 
The proofs of the above results proceed by combining some of the 
ideas in \cite{EL99}, \cite{Bulois-Evain} and \cite{RT22},
and using an induction argument. We assume that 
$S^{[n,m]}$ is irreducible and show that $S^{[n,m,m+1]}$ is irreducible.
Using the surjectivity of the natural map $S^{[n,m,m+1]}\to S^{[n,m+1]}$
we see that $S^{[n,m+1]}$ is irreducible. 

A crucial input in all the 
proofs is that the dimension of some of the spaces of the type $S_p^{[l',l]}$ 
(this notation is explained before Lemma \ref{dim W_i}) satisfy a certain upper bound. 
These dimensions have been computed in \cite{Bulois-Evain}
when $0\leqslant l-l'\leqslant 2$. It is natural to ask if the methods in this 
article can be used to proved the irreduciblity of $S^{[n_1,n_2,n_3]}$ for all triples. 
One of the obstacles is the non-existence of similar bounds on the dimension of $S_p^{[l',l]}$  
for all pairs $(l',l)$ with $l-l'\geqslant0$.\\

\noindent
{\bf Acknowledgments}. We thank the referee for a very 
careful reading of the article and for 
numerous suggestions which helped in improving the exposition.
We thank the organizers of the Virtual Commutative Algebra Seminar 
at IIT Bombay, from where we came to know of this question. 

\section{Preliminaries}\label{section preliminaries}
Let $S$ be a smooth projective surface over $\mb C$.
For a pair of positive integers $n,m$ with $n<m$,
the nested Hilbert scheme $S^{[n,m]}$ parametrizes nested subschemes
$\xi_n \subset \xi_m$ of $S$, where $\xi_i$ is a finite scheme of length $i$.
Recall that the scheme $S^{[n,m]}$ represents the functor of nested flat families 
$\hilb^{[n,m]}_S $ 
$$\hilb^{[n,m]}_S : {\rm Sch}/\C \longrightarrow {\rm Sets}\,,$$
where $\hilb^{[n,m]}_S(T)$ is the set of isomorphism
classes of $T$-flat subschemes $X_n \subset X_m \subset S \times T$
such that for each point $t\in T$, the length of the subscheme
$X_n\otimes k(t)$ is $n$ and the length of the subscheme 
$X_m\otimes k(t)$ is $m$. 
In particular, we have universal nested families of closed 
subschemes $Z_n \subset Z_m \subset S \times S^{[n,m]}$.
The closed points of $Z_n$ and $Z_m$ have the following descriptions:
\begin{align*}
	Z_n &= \{(p,\xi_n,\xi_m) \in S \times S^{[n,m]}\quad \vert 
	\quad p \in \xi_n \subset \xi_m \}\, ,\\
	Z_m &=\{(p,\xi_n,\xi_m) \in S \times S^{[n,m]}\quad \vert 
	\quad  p \in \xi_m \}\,.
\end{align*}

We have the projection map 
$$\pi_m:S^{[n,m]}\longrightarrow S^{[m]}\,.$$
Let $\I_m$ denote the ideal sheaf of the universal subscheme
inside $S\times S^{[m]}$. 
Consider the map 
$${\rm Id}_S\times \pi_m:S\times S^{[n,m]}\longrightarrow S\times S^{[m]}\,.$$
Denote the pullback 
$$\tilde \I_m:=({\rm Id}_S\times \pi_m)^*\I_m\,.$$ 
Consider the projective bundle 
\begin{equation}
	\varphi: \mb P(\tilde \I_m)\longrightarrow S\times S^{[n,m]}\,.
\end{equation}
On $\mb P(\tilde \I_m)$, we have the tautological quotient 
$$\varphi^*\tilde \I_m\longrightarrow \mc O_{\mb P(\tilde \I_m)}(1)\,.$$
Let $\varphi_1$ denote the composite 
$\mb P(\tilde \I_m)\stackrel{\varphi}{\longrightarrow} S\times S^{[n,m]}\longrightarrow S$,
where the second map is the projection to $S$.
Similarly, let $\varphi_2$ denote the composite 
$\mb P(\tilde \I_m)\stackrel{\varphi}{\longrightarrow} S\times S^{[n,m]}\longrightarrow S^{[n,m]}$,
where the second map is the projection to $S^{[n,m]}$.
Consider the graph of $\varphi_1$, 
$$\mb P(\tilde \I_m)\stackrel{\iota}{\hookrightarrow}S\times \mb P(\tilde \I_m)\,.$$
Since $\iota$ is the graph of $\varphi_1$, it follows 
that the composite map 
$\mb P(\tilde \I_m)\stackrel{\iota}{\hookrightarrow}S\times \mb P(\tilde \I_m)\to \mb P(\tilde \I_m)$
is the identity. 
This shows that the sheaf $\iota_*\mc O_{\mb P(\tilde \I_m)}(1)$ 
on $S\times \mb P(\tilde \I_m)$ is flat over $\mb P(\tilde \I_m)$.

Now consider the map 
$$({\rm Id}_S\times \varphi_2):S\times \mb P(\tilde \I_m)\longrightarrow S\times S^{[n,m]}\,.$$
On $S \times \mb P(\tilde \I_m)$, there is a canonical surjection 
$$ \delta : ({\rm Id}_S \times \varphi_2)^*\tilde \I_m
\longrightarrow \iota_*\iota^*({\rm Id}_S \times \varphi_2)^*\tilde \I_m =\iota_*\varphi^*\tilde\I_m
\longrightarrow \iota_*\mc O_{\mb P(\tilde \I_m)}(1)\,.$$
Using $\delta$ we define a sheaf $\T$ on $S\times \mb P(\tilde \I_m)$ by the push-out diagram below
\begin{equation}\label{universal push out}
\begin{tikzcd}
	0 \ar[r]{}  &({\rm Id}_S \times \varphi_2)^* \I_m \ar[r] \arrow{d}{\delta} & 
	\O_{S \times \mb P(\tilde \I_m)} \ar[r]{} \ar[d]{} & 
	({\rm Id}_S \times \varphi_2)^*\O_{Z_m} \ar[d,-,double equal sign distance,double] \ar[r]{} &0\\
	0 \ar[r]  & \iota_*\mc O_{\mb P(\tilde \I_m)}(1)\ar[r] & 
		\T \ar[r]{} &({\rm Id}_S \times \varphi_2)^*\O_{Z_m} \ar[r]{} &0\,.
\end{tikzcd}
\end{equation}

\begin{remark}\label{pullback flat}
	Recall the following general fact. 
	Let $X\to Y$ be a map of schemes and let $\mc F$ be a quasi-coherent
	sheaf on $X$ which is flat over $Y$. Let $f:Y'\to Y$ be a morphism
	of schemes and consider the Cartesian square 
	\[\xymatrix{
		X'\ar[r]^{\tilde f}\ar[d]& X\ar[d]\\
		Y'\ar[r]^f& Y
	}
	\]
	Then one easily checks that the sheaf $\tilde f^*\mc F$ 
	is flat over $Y'$. \hfill\qedsymbol
\end{remark}
Applying Remark \ref{pullback flat}  to the diagram 
\[\xymatrix{
	S\times \mb P(\tilde \I_m)\ar[rr]^{{\rm Id}_S\times \varphi_2}\ar[d]&& S\times S^{[n,m]}\ar[d]\\
	\mb P(\tilde \I_m)\ar[rr]^{\varphi_2}&& S^{[n,m]}
}
\]
and the sheaf $\mc O_{Z_m}$ on $S\times S^{[n,m]}$
we see that $({\rm Id}_S \times \varphi_2)^*\O_{Z_m}$ 
is flat over $\mb P(\tilde \I_m)$. We already saw that 
$\iota_*\mc O_{\mb P(\tilde \I_m)}(1)$ 
is flat over $\mb P(\tilde \I_m)$.
Thus, it follows that the sheaf $\mc T$ on $S\times \mb P(\tilde \I_m)$ 
is flat over $\mb P(\tilde \I_m)$.
It is clear that $\T$ is a family of quotients of length $m+1$. 
This gives a nested family of quotients 
$$\O_{S \times \mb P(\tilde \I_m)} \longrightarrow \T \longrightarrow ({\rm Id}_S \times \varphi_2)^*\O_{Z_m}\longrightarrow 
	({\rm Id}_S \times \varphi_2)^*\O_{Z_n}$$ 
on $S \times \mb P(\tilde \I_m)$. Using the universal property 
for $S^{[n,m+1]}$ and the quotients 
$$\O_{S \times \mb P(\tilde \I_m)} \longrightarrow \T \longrightarrow ({\rm Id}_S \times \varphi_1)^*\O_{Z_n}\,$$
we get a map 
\begin{equation}\label{def psi}
	\psi: \mb P(\tilde \I_m) \longrightarrow S^{[n,m+1]}\,.	
\end{equation}
A pointwise description of this map is given as follows. 
Let $(p,\xi_n,\xi_m) \in S \times S^{[n,m]}$ be a closed point.
So we have a short exact sequence 
$$0 \longrightarrow \ii_{\xi_m} \longrightarrow \O_S \longrightarrow \O_{\xi_m} \longrightarrow 0\,.$$
A point in $\mb P(\tilde \I_m)$ over $(p,\xi_n,\xi_m)$ is given by a quotient 
$\lambda : \ii_{\xi_m} \longrightarrow k(p)$.
We shall represent such a point by the tuple $(p,\xi_n,\xi_m,\lambda)$.
We get the quotient $\O_S \longrightarrow \mc O_{\xi_{m+1}}$ by the push-out diagram below
in which the columns are short exact sequences.
\begin{equation}\label{e4}
\begin{tikzcd}
	 &\ii_{\xi_{m+1}}  \ar[r,-,double equal sign distance,] \arrow{d} & \ii_{\xi_{m+1}}\ar[d]\\
	0 \ar[r]{}  &\ii_{\xi_m}  \ar[r] \arrow{d}{\lambda} & \O_S \ar[r]{} \ar[d]{} & 
	\O_{\xi_m} \ar[d,-,double equal sign distance,double] \ar[r]{} &0\\
	0 \ar[r]  &k(p) \ar[r] &  \mc O_{\xi_{m+1}}\ar[r]{} & \O_{\xi_m} \ar[r]{} &0	
\end{tikzcd}
\end{equation}
The map $\psi$ takes the point 
$(p,\xi_n,\xi_m,\lambda)$ of $\P(\I_m)$ to the point $(p,\xi_n,\xi_{m+1})\in S\times S^{[n,m+1]}$.

We note the following maps
\begin{equation}\label{diagram phi psi}
	\xymatrix{
		\mb P(\tilde \I_m)\ar[r]^{\psi}\ar[d]_{\varphi} & S^{[n,m+1]}\\
		S\times S^{[n,m]}
	}
\end{equation}

For an $\mc O_S$ module $\mc F$, we shall denote by 
$\mc F_p$ the localization $\mc F\otimes_{\mc O_S} \mc O_{S,p}$.
Here $\mc O_{S,p}$ is the local ring of $S$ at the closed point $p$.

\begin{lemma}\label{surjective psi}
	The map $\psi$ is surjective on closed points. 
\end{lemma}
\begin{proof}
	A closed point in $S^{[n,m+1]}$ corresponds to 
	subschemes $\xi_n\subset \xi_{m+1}$ with ${\rm length}(\xi_n)=n$
	and ${\rm length}(\xi_{m+1})=m+1$. Let $K$ denote the kernel
	of the map $\mc O_{\xi_{m+1}}\longrightarrow \mc O_{\xi_n}$. Then we may write 
	$$K=\bigoplus_{p\in {\rm Supp}(K)}K_p\,.$$
	Choose any map $k(p)\longrightarrow K_p$ of $\mc O_{S,p}$ modules and form the 
	diagram 
	\[
	\xymatrix{
		 & k(p)\ar@{=}[r]\ar[d] & k(p)\ar[d]^\lambda\\
		 0\ar[r] & K\ar[r] & \mc O_{\xi_{m+1}}\ar[r]\ar[d]^{\theta} & \mc O_{\xi_n} \ar[r]\ar@{=}[d] & 0\\
		 && \mc O_{\xi_{m}}\ar[r] & \mc O_{\xi_{n}} \ar[r] & 0
	}
	\]
	Note that the middle column is a short exact sequence. 
	Using this observation and applying Snake Lemma to the diagram
	\begin{equation}\label{e3}
	\xymatrix{
		0\ar[r] & \mc I_{\xi_{m+1}}\ar[r]\ar[d] & 
			\mc O_S\ar[r]\ar@{=}[d] & \mc O_{\xi_m+1} \ar[r]\ar[d]^{\theta} & 0\\
		0\ar[r]&\mc I_{\xi_{m}}\ar[r]& \mc O_{S}\ar[r] & \mc O_{\xi_{m}} \ar[r] & 0
	}
	\end{equation}
	one easily concludes that we have a short exact sequence 
	of ideal sheaves 
	$$0\longrightarrow \mc I_{\xi_{m+1}}\longrightarrow \mc I_{\xi_{m}}\stackrel{\lambda}{\longrightarrow} k(p)\longrightarrow 0\,.$$
	The reader will easily check that when we take the push-out of the lower row
	in \eqref{e3} along the map $\lambda$, we get diagram \eqref{e4}.
	One easily concludes that the closed point $(p,\xi_{n},\xi_{m},\lambda)\in \mb P(\tilde \I_m)$
	is mapped to the closed point $(\xi_n,\xi_{m+1})\in S^{[n,m+1]}$ under $\psi$.
	This completes the proof of the Lemma. 
\end{proof}

\section{Irreducibility of $S^{[n,m]}$}

Let $W_{i,[n,m]}$ denote the following locus in $S\times S^{[n,m]}$
$$ W_{i,[n,m]} := \{(p, \xi_n,\xi_m) \in S \times S^{[n,m]}\quad  \vert\quad 
{\rm dim}(\ii_{\xi_m} \otimes k(p))=i \quad \}\,.$$
In other words, it is the locus of points $(p,\xi_n,\xi_m)$
such that the ideal $\mc I_{\xi_m}$ is generated by exactly $i$
elements at the point $p$. Since $\xi_m$ is a zero dimensional scheme
on a smooth surface, it follows that if $p$ is in the support 
of $\xi_m$, then ${\rm dim}(\ii_{\xi_m} \otimes k(p))\geqslant 2$. 
In other words, $W_{1,[n,m]}$ is the complement of the 
universal family $Z_m$ in $S \times S^{[n,m]}$.
For $i\geqslant 2$ we define subsets $W_{i,[n,m],l',l}\subset W_{i,[n,m]}$
as follows. 

\begin{definition}
	Let $i\geqslant 2$. Let $W_{i,[n,m],l',l}\subset W_{i,[n,m]}$ be the subset 
	consisting of points $(p,\xi_n,\xi_m)$ such that 
	${\rm length}(\mc O_{\xi_{n},p})=l'$ and ${\rm length}(\mc O_{\xi_{m},p})=l$.
\end{definition}
Notice that for the set $W_{i,[n,m],l',l}$ to be nonempty we 
need that $0\leqslant l'\leqslant n$, $0\leqslant l'\leqslant l$ and 
$1\leqslant l\leqslant m$. 
As $i\geqslant 2$, we
have that $p\in {\rm Supp}(\xi_m)$, which implies that $1\leqslant l$. 
Note that $l'=0$ is allowed as it may happen 
that $p$ is not in the support of $\xi_n$. 

Clearly, 
\begin{equation}\label{strat W_i}
	W_{i,[n,m]}=\bigcup_{l',l}W_{i,[n,m],l',l}\,.
\end{equation}
In the next lemma, using the sets $W_{i,[n,m],l',l}$, we shall obtain a bound 
on the dimension of $W_{i,[n,m]}$. We need the following notations. 
Let $p\in S$ denote a closed point. 
\begin{itemize}
	\item By $S^{[0,m]}$ we mean $S^{[m]}$.
	\item Let $S^{[l]}_{p,i}$ denote the subset of $S^{[l]}$ corresponding
	to subschemes $\eta$ satisfying the following two conditions:
	${\rm Supp}(\eta)=\{p\}$ and ${\rm dim}(\ii_{\eta}\otimes k(p))=i$.
	\item Let 
	$S^{[l',l]}_{p,i}$
	denote the subset of $S^{[l',l]}$ consisting of pairs $(\xi_{l'},\xi_l)$
	satisfying the following two conditions:
	${\rm Supp}(\xi_l)=\{p\}$ and ${\rm dim}(\ii_{\xi_l}\otimes k(p))=i$.
	\item By $S^{[0,l]}_{p,i}$ we mean $S^{[l]}_{p,i}$.
\end{itemize}

\begin{lemma}\label{dim W_i}
	Fix integers $n<m$. Consider pairs of integers $(l',l)$ for which the following three conditions hold:
	\begin{itemize}
		\item $0\leqslant n-l'\leqslant m-l$, 
		\item $0\leqslant l'\leqslant l$,
		\item $1\leqslant l$.
	\end{itemize}
	Assume that for each such pair, the locus $S^{[n-l',m-l]}$ is irreducible of 
	dimension $2(m-l)$. Let $i\geqslant 2$. 
	Then ${\rm dim}(W_{i,[n,m]})\leqslant 2m+2-i$.
\end{lemma}
\begin{proof}
	In view of \eqref{strat W_i}
	it suffices to show that if $W_{i,[n,m],l',l}$ is nonempty 
	then we have ${\rm dim}(W_{i,[n,m],l',l})\leqslant 2m+2-i$. 
	The argument is similar to that of \cite[Lemma 3.3]{RT22},
	along with a key input from \cite{Bulois-Evain}. Consider the projection
	map $p_1:W_{i,[n,m],l',l}\longrightarrow S$ which sends $(p,\xi_n,\xi_m)\mapsto p$. 
	We shall find an upper bound for the dimension of the fiber over 
	a closed point $p\in S$. Let $U$ denote the open subset $S\setminus \{p\}$.
	Given a point $(p,\xi_n,\xi_m)\in p_1^{-1}(p)$, we may write 
	$$\mc O_{\xi_m}=\mc O_{\xi_m,p}\bigoplus \left(\bigoplus_{q\in U}\mc O_{\xi_m,q}\right)\,,\qquad 
		\mc O_{\xi_n}=\mc O_{\xi_n,p}\bigoplus \left(\bigoplus_{q\in U}\mc O_{\xi_n,q}\right)$$
	The quotient $\mc O_{\xi_m}\longrightarrow \mc O_{\xi_n}$ gives rise to quotients 
	$$\mc O_{\xi_m,p}\longrightarrow \mc O_{\xi_n,p}\,,\qquad \left(\bigoplus_{q\in U}\mc O_{\xi_m,q}\right)\longrightarrow 
		\left(\bigoplus_{q\in U}\mc O_{\xi_n,q}\right)$$
	This gives rise to the following map which is an inclusion on closed points 
	\begin{equation}\label{e1}
		p_1^{-1}(p)\longrightarrow S^{[l',l]}_{p,i}\times U^{[n-l',m-l]}\,.
	\end{equation}
	When $l'=0$ the above map is 
	\begin{equation}\label{e1l'=0}
		p_1^{-1}(p)\longrightarrow S^{[l]}_{p,i}\times U^{[n,m-l]}\,.
	\end{equation}
	As $U^{[n-l',m-l]}$ is an open subset of $S^{[n-l',m-l]}$,
	and the latter is irreducible of dimension $2(m-l)$ by our hypothesis, it follows
	that ${\rm dim}(U^{[n-l',m-l]})=2(m-l)$. Next we a bound on the dimension of 
	$S^{[l',l]}_{p,i}$. To do this we shall first give a bound on the dimension
	of $S^{[l]}_{p,i}$.
	
	First we consider the case $l'\neq 0$.
	Fix a point $\xi_l\in S^{[l]}_{p,i}$. 
	Let $M$ be a module over the local ring $\mc O_{S,p}$
	whose support is zero dimensional.
	By ${\rm Soc}(M)$ we mean the space ${\rm Hom}_{\mc O_{S,p}}(k(p),M)$.
	Since the only closed point in the support of $\xi_l$ is $p$, 
	it follows that if we have a subscheme $\xi_{l-1}\subset \xi_l$,
	then the kernel of the map $\mc O_{\xi_l}\to \mc O_{\xi_{l-1}}$ 
	is isomorphic to $k(p)$. Conversely, taking the quotient of an 
	inclusion of $\mc O_S$ modules 
	$k(p)\to \mc O_{\xi_l}$ gives a length $l-1$ subscheme of $\xi_l$. 
	This shows that the set of subschemes of length $l-1$ of $\xi_l$  
	is in bijective correspondence with $\mb P({\rm Soc}(\mc O_{\xi_l})^\vee)$.
	By \cite[Lemma 2]{EL99}, we have ${\rm dim}(\mb P({\rm Soc}(\mc O_{\xi_l})^\vee))=i-2$. 
	Thus, all fibers of the map $S^{[l-1,l]}_{p,i}\to S^{[l]}_{p,i}$ have 
	dimension $i-2$. 
	From this, it follows that 
	$${\rm dim}(S^{[l-1,l]}_{p,i})={\rm dim}(S^{[l]}_{p,i})+i-2\,.$$
	As $S^{[l-1,l]}_{p,i}\subset S^{[l-1,l]}_{p}$
	it follows that ${\rm dim}(S^{[l-1,l]}_{p,i})\leqslant
	{\rm dim}(S^{[l-1,l]}_{p})$.
	In \cite[Corollary 5.9]{Bulois-Evain}, it is proved that 
	${\rm dim}(S^{[l-1,l]}_{p})=l-1$. Thus, we get
	$${\rm dim}(S^{[l]}_{p,i})+i-2={\rm dim}(S^{[l-1,l]}_{p,i})\leqslant
		{\rm dim}(S^{[l-1,l]}_{p})= l-1\,.$$
	The above gives the following bound on the dimension of $S^{[l]}_{p,i}$, 
	\begin{equation}\label{e2}
		{\rm dim}(S^{[l]}_{p,i})\leqslant l-i+1\,.
	\end{equation}
	
	The natural map $S^{[l',l]}_{p,i}\longrightarrow S^{[l]}_{p,i}\times S^{[l']}_p$
	is an inclusion on closed points. As $l'\geqslant 1$, we have 
	${\rm dim}(S^{[l']}_p)=l'-1$, 
	see \cite{Br77}. Thus, it follows that 
	\begin{equation}\label{dim S_p,i^l',l}
	{\rm dim}(S^{[l',l]}_{p,i})\leqslant l-i+1+l'-1=l+l'-i\leqslant 2l-i\,.
	\end{equation}
	Thus, using \eqref{e1} it follows that 
	$${\rm dim}(p_1^{-1}(p))\leqslant 2l-i+2(m-l)=2m-i\,,$$
	from which it follows that 
	$${\rm dim}(W_{i,[n,m],l',l})\leqslant 2m+2-i\,.$$
	
	Next we consider the case $l'=0$. Using \eqref{e1l'=0} and \eqref{e2}
	we get 
	$${\rm dim}(p_1^{-1}(p))\leqslant 2(m-l)+l-i+1=2m-l-i+1\,.$$
	It follows that 
	$${\rm dim}(W_{i,[n,m],0,l})\leqslant 2m-l-i+3\,.$$
	In the proof of \cite[Lemma 3.2]{RT22} it is proved that 
	$l\geqslant {i \choose 2}$. 
	Since $i\geqslant 2$, we have that ${i \choose 2}-1\geqslant 0$. Thus, we get
	$${\rm dim}(W_{i,[n,m],0,l})\leqslant 2m-l-i+3\leqslant 2m +2-i -{i \choose 2}+1
	\leqslant 2m+2-i\,.$$
	This completes the proof 
	of the Lemma.
\end{proof}

\begin{theorem}\label{main-theorem}
	Let $n$ and $m$ be two positive integers such that $n < m$. 
	Then $S^{[n,m,m+1]}$ and $S^{[n,m]}$ are irreducible. 
\end{theorem}
\begin{proof}
	Let $\mc A$ be the
	set of pairs of integers $(a,b)$ with $1\leqslant a<b$ 
	and $S^{[a,b]}$ reducible. Assume $\mc A$ is nonempty. By \cite[Corollary 7.3]{Fo}
	for every $b\geqslant 2$ the pair $(1,b)\notin \mc A$. Similarly, by 
	\cite[Theorem 3.0.1]{Cheah98} for every $a\geqslant 1$ the  pair
	$(a,a+1)\notin \mc A$. Consider the projection map 
	to the first coordinate $\mc A\longrightarrow \Z_{\geqslant 1}$,
	where $\Z_{\geqslant 1}$ denotes the set of positive integers. 
	Let $n$ be the smallest integer in the image of this map. 
	Clearly, $n>1$. Among the set of integers $b$ such that $(n,b)\in \mc A$,
	let $b_0$ be the smallest. Clearly, $b_0>n+1$. Let $m=b_0-1$. 
	Then $m\geqslant n+1$. We conclude that for all pairs of integers
	$(a,b)$ with $1\leqslant a<b$, if $a<n$ then $S^{[a,b]}$ is irreducible,
	and for all integers $b$ such that $n<b\leqslant m$, $S^{[n,b]}$ is irreducible. 
	Further $S^{[n,m+1]}$ 
	is reducible. We will arrive at a contradiction, which will prove
	that $\mc A$ is empty, and hence prove the theorem.
	
	Note that if $S^{[a,b]}$ is irreducible
	then its dimension is $2b$. This can be seen as follows.
	Consider the open subset consisting of pairs $(\xi_a,\xi_b)$ 
	such that the support of $\xi_b$ has $b$ distinct points. 
	The natural map from this open set to $S^{[b]}$ is dominant and quasi-finite and so 
	this open set has dimension $2b$. Since $S^{[a,b]}$ is irreducible, it follows
	that it has dimension $2b$.

	The method of proof is identical to the method in \cite[Proposition 5]{EL99}. 
	Consider the map $\varphi$ in \eqref{diagram phi psi}.
	We claim that we can find locally free sheaves $\mc F$ of rank $r$ 
	and $\mc E$ of rank $r+1$ on $S\times S^{[n,m]}$
	which fit into a short exact sequence  
	$$0\longrightarrow \mc F\longrightarrow \mc E\longrightarrow \tilde \I_m\longrightarrow 0\,$$
	on $S\times S^{[n,m]}$. Let $\mc E$ be a locally free sheaf which surjects 
	onto $\tilde \I_m$ and let $\mc F$ be the kernel of this surjection. 
	As $\tilde \I_m$ is flat over $S^{[n,m]}$ and $\mc E$ is flat, it follows
	that $\mc F$ is flat over $S^{[n,m]}$. If $x\in S^{[n,m]}$ is a closed point,
	then the restriction to $S\times x$ gives a short exact sequence 
	$$0\longrightarrow \mc F\vert_{S\times x}\longrightarrow \mc E\vert_{S\times x}\longrightarrow 
		\tilde \I_m\vert_{S\times x}\longrightarrow 0\,.$$
	As $\tilde \I_m\vert_{S\times x}$ is the ideal sheaf of a zero dimensional 
	scheme, it follows this has projective dimension 1. Thus, it follows
	that $\mc F\vert_{S\times x}$ is locally free on $S$. 
	Using the following result from commutative algebra,
	we see that $\mc F$ is locally free. Let $A \to B$ be a local 
	homomorphism of local rings, $M$ a finite $B$ module which is flat over $A$ and 
	$M/(\mf m_A M)$ is a free $B/(\mf m_A B)$ module. Then $M$ is a free $B$ module.
	It is clear that if $\mc F$ has rank $r$ then $\mc E$ has rank $r+1$. This 
	completes the proof of the claim.

	Let $X$ be a scheme and suppose that 
	$0\to \mc A\to \mc B\to \mc C\to 0$
	is a short exact sequence of coherent sheaves on $X$.
	Let ${\rm Sym}^*(\mc B)$ denote the sheaf of algebras on $X$
	associated to $\mc B$. Let $\mc J\subset {\rm Sym}^*(\mc B)$ 
	denote the sheaf of ideals generated by $\mc A$. Then we have 
	$${\rm Sym}^*(\mc C)={\rm Sym}^*(\mc B)/\mc J\,.$$
	On ${\rm Proj}({\rm Sym}^*(\mc B))\xrightarrow{\pi}X$ we have the map 
	of sheaves $\pi^*\mc B\to \mc O(1)$.
	The sheaf of ideals of 
	${\rm Proj}({\rm Sym}^*(\mc C))\subset {\rm Proj}({\rm Sym}^*(\mc B))$
	is the image of the composite $\pi^*\mc A\to \pi^*\mc B\to \mc O(1)$.
	
	Let $\pi:\mb P(\mc E)\longrightarrow S\times S^{[n,m]}$
	denote the projective bundle.
	It follows that $\mb P(\tilde \I_m)\subset \mb P(\mc E)$
	is the vanishing locus of the composite homomorphism 
	$\pi^*\mc F\longrightarrow \pi^*\mc E\longrightarrow \mc O_{\mb P(\mc E)}(1)$.
	As $S\times S^{[n,m]}$ is irreducible, it follows 
	that $\mb P(\mc E)$ is irreducible of dimension $2m+2+r$. 
	As $\mb P(\tilde \I_m)$ is locally cut out by $r$ 
	equations, it follows that each irreducible component 
	of $\mb P(\tilde \I_m)$ has dimension at least $2m+2$.
	
	Let $i\geqslant 2$. The hypothesis of Lemma \ref{dim W_i} holds and so we 
	get that ${\rm dim}(W_{i,[n,m]})\leqslant 2m+2-i$. The dimension 
	of the fiber of $\varphi:\mb P(\tilde \I_m)\longrightarrow S\times S^{[n,m]}$
	over a point $(p,\xi_n,\xi_m)\in W_{i,[n,m]}$ is $i-1$. Thus, 
	$${\rm dim}(\varphi^{-1}(W_{i,[n,m]}))\leqslant 2m+2-i+i-1=2m+1\,.$$
	Let $T$ be an irreducible component of $\mb P(\tilde \I_m)$. 
	As ${\rm dim}(T)\geqslant 2m+2$, it follows that $T$ cannot be 
	contained in $\varphi^{-1}(W_{i,[n,m]})$ for any $i\geqslant 2$. 
	Thus, $T$ meets the set $\varphi^{-1}(W_{1,[n,m]})$. 
	Note that $W_{1,[n,m]}$ is the complement of $Z_m$ and so is an 
	open subset of $S\times S^{[n,m]}$. Moreover, it is clear that 
	$$\varphi: \varphi^{-1}(W_{1,[n,m]}) \longrightarrow W_{1,[n,m]}$$
	is an isomorphism. Let $\widetilde W_1$ denote the open and irreducible 
	subset $\varphi^{-1}(W_{1,[n,m]})$. It follows that $T\cap \widetilde W_1$
	is open in $T$ and so is also dense in $T$. It follows that 
	$T$ is contained in the closure of $\widetilde W_1$. 
	Thus, every irreducible component is contained in the closure 
	of $\widetilde W_1$. As $\widetilde W_1$ is irreducible, so is its closure. 
	It follows
	that every irreducible component of $\mb P(\tilde \I_m)$ is contained in the 
	closure of $\widetilde W_1$. Thus, there is only one irreducible component,
	that is, $\mb P(\tilde \I_m)$ is irreducible. 
	
	We saw in Lemma \ref{surjective psi} that $\psi$ is surjective.
	It follows that $S^{[n,m+1]}$ is irreducible. This is a contradiction and so 
	$\mc A$ is empty. 
	As $\mb P(\tilde \I_m)\cong S^{[n,m,m+1]}$,
	the above discussion also shows that $S^{[n,m,m+1]}$ is irreducible.
	The completes the proof.
\end{proof}

\section{Irreducibility of $S^{[n,n+1,m]}$}
For a tuple of positive integers $a,b,c$ with $a<b<c$,
the nested Hilbert scheme $S^{[a,b,c]}$ parametrizes nested closed subschemes
$\xi_a \subset \xi_b \subset \xi_c$ of $S$,
where $\xi_i$ is a finite scheme of length $i$.
We have the universal nested family of closed 
subschemes $Z_c \subset S \times S^{[a,b,c]}$.
The closed points of $Z_c$ have the following descriptions.
\begin{align*}
    Z_c &=\{(p,\xi_a,\xi_b, \xi_c) \in S \times S^{[a,b,c]}\,\, \vert 
	\,\,  p \in \xi_c \}\,.
\end{align*}

We have the projection map 
$$\pi_c:S^{[a,b,c]}\longrightarrow S^{[c]}\,.$$
Let $\I_c$ denote the ideal sheaf of the universal subscheme
inside $S\times S^{[c]}$. Consider the map 
$${\rm Id}_S\times \pi_c:S\times S^{[a,b,c]}\longrightarrow S\times S^{[c]}\,.$$
Denote the pullback 
$$\tilde \I_c:=({\rm Id}_S\times \pi_c)^*\I_c\,.$$ 
Consider the projective bundle 
\begin{equation}
	\varphi: \mb P(\tilde \I_c)\longrightarrow S\times S^{[a,b,c]}\,.
\end{equation}
We define the map $\psi : \P(\tilde \I_c) \longrightarrow S^{[a,b,c+1]}$
in the same way as defined in \eqref{def psi} in \S\ref{section preliminaries}.
We have the following maps
\begin{equation}\label{phi psi a,b,c}
	\xymatrix{
		\mb P(\tilde \I_c)\ar[r]^{\psi}\ar[d]_{\varphi} & S^{[a,b,c+1]}\\
		S\times S^{[a,b,c]}
	}
\end{equation}
The pointwise description of the map $\psi$ is similar to the one 
given in \S\ref{section preliminaries} and is left to the reader.
By similar argument as in the proof of Lemma \ref{surjective psi}, 
we conclude that the map $\psi$ is surjective on closed points.

As in the case of $S^{[n,m]}$, here also we define the subsets $W_{i,[n,n+1,m]}$ in a similar manner.
Let $W_{i,[a,b,c]}$ denote the locus in $S \times S^{[a,b,c]}$
where the ideal sheaf $\I_c$ of $Z_c$ is generated by $i$ elements,
that is,
$$ W_{i,[a,b,c]} := \{(p, \xi_a,\xi_b,\xi_c) \in S \times S^{[a,b,c]}\,\,  \vert \,\, 
{\rm dim}(\ii_{\xi_c} \otimes k(p))=i \,\, \}\,.$$
The set $W_{1,[a,b,c]}$ is the complement of the 
universal family $Z_c$ in $S \times S^{[a,b,c]}$.
Define subsets $W_{i,[a,b,c],l'',l',l}\subset W_{i,[a,b,c]}$
as follows. 

\begin{definition}
	Let $i\geqslant 2$. Let $W_{i,[a,b,c],l'',l',l}\subset W_{i,[a,b,c]}$ be the subset 
	consisting of points $(p,\xi_a,\xi_b,\xi_c)$ such that 
	${\rm length}(\O_{\xi_{a},p})=l''$, ${\rm length}(\O_{\xi_{b},p})=l'$ and 
	${\rm length}(\O_{\xi_{c},p})=l$.
\end{definition}
Notice that for the set $W_{i,[a,b,c],l'',l',l}$ to be nonempty we 
need that $0\leqslant a-l''\leqslant b-l'\leqslant c-l$, 
$0\leqslant l''\leqslant l'\leqslant l$
and $1\leqslant l$. 
As $i\geqslant 2$, we
have that $p\in {\rm Supp}(\xi_m)$, which implies that $1\leqslant l$. 
Clearly, 
\begin{equation}\label{strat W_i_a,b,c}
	W_{i,[a,b,c]}=\bigcup_{l,l',l''}W_{i,[a,b,c],l'',l',l}\,.
\end{equation}

Let $p\in S$ denote a closed point. 
Let
$S^{[l'',l',l]}_{p,i}$
denote the subset of $S^{[l'',l',l]}$ consisting of the tuples $(\xi_{l''},\xi_{l'},\xi_l)$
satisfying the following two conditions:
${\rm Supp}(\xi_l)=\{p\}$ and ${\rm dim}(\ii_{\xi_l}\otimes k(p))=i$.

\begin{lemma}\label{dimension W_i,n+1}
Let $n$ and $m$ be two positive integers such that $n+1 < m$.
Consider triples of integers $(l'',l',l)$ which satisfy the following three conditions 
\begin{itemize}
	\item $0\leqslant n-l''\leqslant n+1-l'\leqslant m-l$, 
	\item $0 \leqslant l'' \leqslant l'\leqslant l$, and 
	\item $1 \leqslant l$.
\end{itemize}
Assume that $S^{[n-l'',n+1-l',m-l]}$ is irreducible of dimension
$2(m-l)$ for all such triples.
Let $i \geqslant 2$.
Then $\dim ( W_{i,[n,n+1,m]}) \leqslant 2m+2-i$.
\end{lemma}
\begin{proof}
It suffices to prove that for $i\geqslant 2$, 
if $W_{i,[n,n+1,m],l'',l',l}$ is nonempty then 
$$\dim (W_{i,[n,n+1,m],l'',l',l}) \leqslant 2m+2-i\,.$$
The proof is very similar to the proof of Lemma \ref{dim W_i}
and so we omit some details. 
Consider  
$p_1 :  W_{i,[n,n+1,m],l'',l',l} \longrightarrow S$
which sends $(p,\xi_n,\xi_{n+1},\xi_m)$ to $p$.
We find an upper bound for the dimension of the fiber over a closed point $p \in S$.
Let $U$ be the open subset $S \setminus \{p\}$.
Given a point $(p,\xi_n,\xi_{n+1},\xi_m) \in p_1^{-1}(p)$,
the quotient $\O_{\xi_m}\longrightarrow \O_{\xi_{n+1}}$ gives rise to quotients 
$$\O_{\xi_m,p}\longrightarrow \O_{\xi_{n+1},p}\,,\qquad \left(\bigoplus_{q\in U} \O_{\xi_m,q}\right)\longrightarrow \left(\bigoplus_{q\in U} \O_{\xi_{n+1},q}\right)$$
and the quotient $\O_{\xi_{n+1}}\longrightarrow \O_{\xi_n}$
gives rise to quotients 
$$\O_{\xi_{n+1},p} \longrightarrow \O_{\xi_n,p}\,,\qquad \left(\bigoplus_{q\in U} \O_{\xi_{n+1},q}\right) \longrightarrow \left(\bigoplus_{q\in U} \O_{\xi_n,q}\right)\,.$$
This gives rise to the following map which is an inclusion on closed points
\begin{equation}
p_1^{-1} (p) \longrightarrow S_{p,i}^{[l'',l',l]} \times U^{[n-l'',n+1-l',m-l]}\,.
\end{equation}
We note that $n+1-l'\geqslant n-l''$, that is,
$l'\leqslant l''+1$.
As $l''\leqslant l'$, there are only the following two possibilities: 
either $l'=l''$
or $l'=l''+1$. 

If $l'=l''$ then by our hypothesis 
$S^{[n-l'',n+1-l',m-l]}$ is irreducible of dimension $2(m-l)$.
If $l'=l''+ 1$ then $S^{[n-l'',n+1-l',m-l]}$ is same as 
$S^{[n+1-l',m-l]}$, which is irreducible of dimension $2(m-l)$
by Theorem \ref{main-theorem}.
So it follows that $\dim U^{[n-l'',n+1-l',m-l]}=2(m-l)$.

Now we need to find an upper bound of $\dim (S_{p,i}^{[l'',l',l]})$.
We have two cases: $l''= l'-1$ and $l''=l'$.
We first consider the case $l''=l'-1$.
There is a natural map 
$$S_{p,i}^{[l'',l',l]} \longrightarrow S_{p,i}^{[l]} \times S_{p}^{[l'',l']}$$
which is an inclusion on closed points.
As $l''=l'-1$, by \cite[Corollary 5.9]{Bulois-Evain}
we have $\dim (S^{[l'',l']}_p) = l'-1$.
Also from \eqref{e2}, we get 
$\dim (S_{p,i}^{l}) \leqslant l+1-i$.
So it follows that 
$$\dim S_{p,i}^{[l'',l',l]} \leqslant (l+1-i) + (l'-1) \leqslant 2l-i \,.$$
This gives 
$$\dim(p_1^{-1}(p)) \leqslant 2(m-l) + 2l-i=2m-i\,.$$
Thus, we get  
$$\dim (W_{i,[n,n+1,m],l'',l',l}) \leqslant 2m+2-i\,.$$

Next we consider the case $l''=l'$.
In this case $S_{p,i}^{[l'',l',l]}$ is same as $S_{p,i}^{[l',l]}$ which 
has dimension at most $2l-i$ by \eqref{dim S_p,i^l',l}.
Thus again we get 
$$\dim (W_{i,[n,n+1,m],l'',l',l}) \leqslant 2m+2-i\,.$$
This proves the lemma.
\end{proof}

\begin{theorem}\label{irreducible n,n+1,m}
Let $n$ and $m$ be two positive integers such that $n+1 < m$.
Then $S^{[n,n+1,m,m+1]}$ and $S^{[n,n+1,m]}$ is irreducible.
\end{theorem}
\begin{proof}
    We follow the same method as we used in the proof of Theorem \ref{main-theorem}.
    Let $\mc A$ be the
	set of pairs of integers $(a,b)$ with $1\leqslant a$, $a+1<b$ 
	and $S^{[a,a+1,b]}$ reducible. 
	Assume that $\mc A$ is nonempty. 
	By \cite[Theorem 3.10]{RT22}
	for every $a \geqslant 1$ the pair $(a,a+2)\notin \mc A$. Consider the projection map 
	to the first coordinate $\mc A\longrightarrow \Z_{\geqslant 1}$. 
	Let $n$ be the smallest integer in the image of this map. 
	Among the set of integers $b$ such that $(n,b)\in \mc A$,
	let $b_0$ be the smallest. Clearly, $b_0>n+2$. Let $m=b_0-1$. 
	Then $m\geqslant n+2$. 
	We conclude that for all pairs of integers
	$(a,b)$ with $1\leqslant a$, $a+1<b$, if $a<n$ then $S^{[a,a+1,b]}$ is irreducible
	and $S^{[n,n+1,b]}$ is irreducible if $b\leqslant m$. 
	Further $S^{[n,n+1,m+1]}$ 
	is reducible. Note that if $S^{[a,a+1,b]}$ is irreducible
	then its dimension is $2b$.
	A similar argument as in the proof of Theorem \ref{main-theorem}, 
	after replacing Lemma \ref{dim W_i} with Lemma \ref{dimension W_i,n+1},
	concludes the proof of the Theorem. 
\end{proof}

\section{Irreducibility of $S^{[n,n+2,m]}$}
We begin with the following Lemma.
\begin{lemma}\label{dimension W_i,n+2}
Fix integers  $1 \leqslant n$ and $n+2<m$.
Consider triples of integers $(l'',l',l)$ which satisfy the following three conditions 
\begin{itemize}
	\item $0\leqslant n-l''\leqslant n+2-l'\leqslant m-l$, 
	\item $0 \leqslant l'' \leqslant l'\leqslant l$, and 
	\item $1 \leqslant l$.
\end{itemize}
Assume that $S^{[n-l'',n+2-l',m-l]}$ is irreducible of dimension
$2(m-l)$ for all such triples. 
Let $i \geqslant 2$.
Then $\dim ( W_{i,[n,n+2,m]}) \leqslant 2m+2-i$.
\end{lemma}
\begin{proof}
From \eqref{strat W_i_a,b,c}, we have,
$$W_{i,[n,n+2,m]}=\bigcup_{l'',l',l}W_{i,[n,n+2,m],l'',l',l}\,.$$
So it suffices to prove that 
$\dim (W_{i,[n,n+2,m],l'',l',l}) \leqslant 2m+2-i$ for $i \geqslant 2$.
Consider the projection $p_1 :  W_{i,[n,n+2,m],l'',l',l} \longrightarrow S$
which sends $(p,\xi_n,\xi_{n+2},\xi_m)$ to $p$.
We find an upper bound for the dimension of the fiber over a closed point $p \in S$.
Let $U$ be the open subset $S \setminus \{p\}$.
Given a point $(p,\xi_n,\xi_{n+2},\xi_m) \in p_1^{-1}(p)$,
the quotient $\O_{\xi_m}\longrightarrow \O_{\xi_{n+2}}$ gives rise to quotients 
$$\O_{\xi_m,p}\longrightarrow \O_{\xi_{n+2},p}\,,\qquad \left(\bigoplus_{q\in U} \O_{\xi_m,q}\right)\longrightarrow \left(\bigoplus_{q\in U} \O_{\xi_{n+2},q}\right)$$
and the quotient $\O_{\xi_{n+2}}\longrightarrow \O_{\xi_n}$
gives rise to quotients 
$$\O_{\xi_{n+2},p} \longrightarrow \O_{\xi_n,p}\,,\qquad \left(\bigoplus_{q\in U} \O_{\xi_{n+2},q}\right) \longrightarrow 
	\left(\bigoplus_{q\in U} \O_{\xi_n,q}\right)\,.$$
This gives rise to the following map which is an inclusion on closed points
\begin{equation*}
p_1^{-1} (p) \longrightarrow S_{p,i}^{[l'',l',l]} \times U^{[n-l'',n+2-l',m-l]}\,.
\end{equation*}
We note that $n+2-l'\geqslant n-l''$, that is,
$l'\leqslant l''+2$.
As $l''\leqslant l'$, there are only the following three possibilities: 
$l'=l''$ or $l'=l''+1$ or $l'=l''+2$. 

If $l''=l'$ then by our hypothesis 
$S^{[n-l'',n+2-l',m-l]}$ is irreducible of dimension $2(m-l)$.
If $l''=l'- 1$ then $S^{[n-l'',n+2-l',m-l]}$ is the same as $S^{[n-l'',n+1-l'',m-l]}$,
which is irreducible of dimension $2(m-l)$
by Theorem \ref{irreducible n,n+1,m}.
If $l''=l'-2$ then $S^{[n-l'',n+2-l',m-l]}$ is same as
$S^{[n+2-l',m-l]}$, 
which is irreducible of dimension $2(m-l)$
by Theorem \ref{main-theorem}.
So it follows that $\dim (U^{[n-l'',n+2-l',m-l]})=2(m-l)$.

Now we need to find an upper bound of $\dim (S_{p,i}^{[l'',l',l]})$.
We have three cases : $l''= l'-2$, $l'' = l'-1$ and $l''=l'$.
We first consider the cases $l''=l'-2$ or $l'-1$.
There is a natural map 
$$S_{p,i}^{[l'',l',l]} \longrightarrow S_{p,i}^{[l]} \times S_{p}^{[l'',l']}$$
which is an inclusion on closed points.
If $l''=l'-2$ then we use \cite[Corollary 7.5]{Bulois-Evain},
and if $l''=l'-1$ then we use \cite[Corollary 5.9]{Bulois-Evain},
to conclude $\dim(S^{[l'',l']}_p) = l'-1$.
Also from \eqref{e2}, we get 
$\dim(S_{p,i}^{[l]}) \leqslant l+1-i$.
So it follows that 
$$\dim S_{p,i}^{[l'',l',l]} \leqslant (l+1-i) + (l'-1) \leqslant 2l-i \,.$$
This gives 
$$\dim(p_1^{-1}(p)) \leqslant 2(m-l) + 2l-i=2m-i\,.$$
Thus we get  
$$\dim (W_{i,[n,n+2,m],l'',l',l}) \leqslant 2m+2-i\,.$$

Next we consider the case $l''=l'$.
In this case $S_{p,i}^{[l'',l',l]}$ is same as $S_{p,i}^{[l',l]}$ which 
has dimension at most $2l-i$ by \eqref{dim S_p,i^l',l}.
Thus again we get 
$$\dim(p_1^{-1}(p)) \leqslant 2(m-l) + 2l-i=2m-i$$
and hence
$$\dim (W_{i,[n,n+2,m],l'',l',l}) \leqslant 2m+2-i\,.$$
This proves the lemma.
\end{proof}

\begin{theorem}\label{irreducible n+2}
Let $n$ and $m$ be two positive integers such that $n+2 < m$.
Then $S^{[n,n+2,m,m+1]}$ and $S^{[n,n+2,m]}$ are irreducible.
\end{theorem}
\begin{proof}
    We follow the same method as we used in proof of Theorem \ref{main-theorem}.
    Let $\mc A$ be the
	set of pairs of integers $(a,b)$ with $1\leqslant a$, $a+2<b$ 
	and $S^{[a,a+2,b]}$ reducible. 
	We prove that $\mc A$ is empty.
	Taking $(n,m)=(a,a+2)$ in Theorem \ref{irreducible n,n+1,m}
	shows that $S^{[a,a+1,a+2,a+3]}$ is irreducible and so it follows 
	that $S^{[a,a+2,a+3]}$ is irreducible. Thus, 
	it follows that
	for every $a \geqslant 1$ the pair $(a,a+3)\notin \mc A$. Consider the projection map 
	to the first coordinate $\mc A\longrightarrow \Z_{\geqslant 1}$. 
	Let $n$ be the smallest integer such in the image of this map. 
	Among the set of integers $b$ such that $(n,b)\in \mc A$,
	let $b_0$ be the smallest. 
	Clearly, $b_0>n+3$. Let $m=b_0-1$. 
	Then $m\geqslant n+3$. 
	We conclude that for all pairs of integers
	$(a,b)$ with $1\leqslant a$, $a+2<b$, if $a<n$ then $S^{[a,a+2,b]}$ is irreducible,
	and $S^{[n,n+2,b]}$ is irreducible if $b\leqslant m$. 
	Further $S^{[n,n+2,m+1]}$ is reducible.
	A similar argument as in the proof of Theorem \ref{main-theorem}, 
	after replacing Lemma \ref{dim W_i} with Lemma \ref{dimension W_i,n+2},
	concludes the proof of the Theorem. 
\end{proof}


\begin{thebibliography}{Add16}
	\expandafter\ifx\csname url\endcsname\relax
	\def\url#1{\texttt{#1}}\fi
	\expandafter\ifx\csname doi\endcsname\relax
	\def\doi#1{\burlalt{doi:#1}{http://dx.doi.org/#1}}\fi
	\expandafter\ifx\csname urlprefix\endcsname\relax\def\urlprefix{URL }\fi
	\expandafter\ifx\csname href\endcsname\relax
	\def\href#1#2{#2}\fi
	\expandafter\ifx\csname burlalt\endcsname\relax
	\def\burlalt#1#2{\href{#2}{#1}}\fi
	
	\bibitem[Add16]{Addington}
	Nicolas Addington.
	\newblock New derived symmetries of some hyperk\"{a}hler varieties.
	\newblock {\em Algebr. Geom.}, 3(2):223--260, 2016.
	\newblock \doi{10.14231/AG-2016-011}.
	
	\bibitem[BE16]{Bulois-Evain}
	Micha\"{e}l Bulois and Laurent Evain.
	\newblock Nested punctual {H}ilbert schemes and commuting varieties of
	parabolic subalgebras.
	\newblock {\em J. Lie Theory}, 26(2):497--533, 2016,
	\burlalt{arXiv:1306.4838v2}{http://arxiv.org/abs/arXiv:1306.4838v2}.
	
	\bibitem[Bri77]{Br77}
	Jo\"{e}l Brian\c{c}on.
	\newblock Description de {$H{\rm ilb}^{n}C\{x,y\}$}.
	\newblock {\em Invent. Math.}, 41(1):45--89, 1977.
	\newblock \doi{10.1007/BF01390164}.
	
	\bibitem[Che98]{Cheah98}
	Jan Cheah.
	\newblock Cellular decompositions for nested {H}ilbert schemes of points.
	\newblock {\em Pacific J. Math.}, 183(1):39--90, 1998.
	\newblock \doi{10.2140/pjm.1998.183.39}.
	
	\bibitem[EL99]{EL99}
	Geir Ellingsrud and Manfred Lehn.
	\newblock Irreducibility of the punctual quotient scheme of a surface.
	\newblock {\em Ark. Mat.}, 37(2):245--254, 1999.
	\newblock \doi{10.1007/BF02412213}.
	
	\bibitem[Fog68]{fo68}
	John Fogarty.
	\newblock Algebraic families on an algebraic surface.
	\newblock {\em Amer. J. Math.}, 90:511--521, 1968.
	\newblock \doi{10.2307/2373541}.
	
	\bibitem[Fog73]{Fo}
	J.~Fogarty.
	\newblock Algebraic families on an algebraic surface. {II}. {T}he {P}icard
	scheme of the punctual {H}ilbert scheme.
	\newblock {\em Amer. J. Math.}, 95:660--687, 1973.
	\newblock \doi{10.2307/2373734}.
	
	\bibitem[GH04]{GH04}
	J\o rgen~Anders Geertsen and Andr\'{e} Hirschowitz.
	\newblock On the stratification of nested {H}ilbert schemes.
	\newblock {\em Comm. Algebra}, 32(8):3025--3041, 2004.
	\newblock \doi{10.1081/AGB-120039277}.
	
	\bibitem[RS21]{RS21}
	Ritvik Ramkumar and Alessio Sammartano.
	\newblock Rational singularities of nested hilbert schemes, 2021.
	\newblock \doi{10.48550/arxiv.2109.09002}.
	
	\bibitem[RT22]{RT22}
	Tim Ryan and Gregory Taylor.
	\newblock Irreducibility and singularities of some nested {H}ilbert schemes.
	\newblock {\em J. Algebra}, 609:380--406, 2022.
	\newblock \doi{10.1016/j.jalgebra.2022.05.037}.
	
\end{thebibliography}
\end{document}